\documentclass[11pt]{amsart}
\usepackage{amscd, amsmath, amsthm, amssymb, xy, xypic, dsfont}
\usepackage{graphicx}

\usepackage[backref=page]{hyperref}
\renewcommand*{\backref}[1]{}
\renewcommand*{\backrefalt}[4]{%
	\ifcase #1 (Not cited.)%
	\or        (Cited on page~#2.)%
	\else      (Cited on pages~#2.)%
	\fi}

\hypersetup{
	colorlinks   = true,
	citecolor    = magenta,
	linkcolor    =blue,
	urlcolor     =magenta	
}

\newtheorem{theorem}{Theorem}[section]
\newtheorem{lemma}[theorem]{Lemma}
\newtheorem{proposition}[theorem]{Proposition}

\theoremstyle{definition}     

\newtheorem{remark}[theorem]{Remark}
\numberwithin{equation}{section}

\def \hd #1 {\bfseries #1  \mdseries}
\def \italic #1 {\bfseries \it #1 \rm \mdseries}
\def \ra {\rightarrow}
\def \cen #1 { \begin{center} #1 \end{center}}

\def \mbz {\mathbb Z}

\def \mbc {\mathbb C}
\def \mbp {\mathbb P}

\def \mco  {\mathcal {O}}

\def \mcX {\mathcal {X}}

\def \Pic {{\rm{Pic}}}

\begin{document}
\title
 {An example of non-K\"ahler Calabi--Yau fourfold}

\author{Nam-Hoon Lee}
\address{
Department of Mathematics Education, Hongik University
42-1, Sangsu-Dong, Mapo-Gu, Seoul 121-791, Korea}
\address{
School of Mathematics, Korea Institute for Advanced Study, Dongdaemun-gu, Seoul 130-722, South Korea }
\email{nhlee@kias.re.kr}
\begin{abstract}
We show that there exists a non-K\"ahler Calabi--Yau fourfold, constructing  an example by smoothing a normal crossing variety.
\end{abstract}
\maketitle
\section{Introduction}
In this note, a \emph{Calabi--Yau manifold} is a  simply-connected compact complex manifold with trivial canonical
class and $H^i(M, \mco_M) =0$ for $0< i < \dim M$.
$K3$ surfaces are  Calabi–Yau twofolds in this definition and they are all  K\"ahler (\cite{Si}), forming  a single irreducible smooth family of dimension $20$ (\cite{Ko, KoNiSp}). Hence, they are all diffeomorphic.
The moduli spaces of projective $K3$ surfaces are a countable union of analytic subspaces  inside of the family of all $K3$ surfaces.

In higher dimensions, the situation is a bit different.  K\"ahler Calabi--Yau manifolds are necessarily projective when their dimensions are greater than two and there exist non-K\"ahler Calabi--Yau threefolds. A number of non-homeomorphic K\"ahler Calabi--Yau threefolds have been constructed (\cite{ AlGrHe, Ba, KrSk})
and it is still an open problem whether there are  finitely many non-homeomorphic  K\"ahler Calabi--Yau threefolds or not. On the other hand, there are found {infinitely} many non-K\"ahler, non-homeomorphic  Calabi--Yau threefolds (\cite{Cl, Fr1, Fr2, HaSa}).
In an effort to understand the situation in a similar way to that of the $K3$ surfaces,
Reid conjectured that there still may be a single irreducible moduli space of  Calabi--Yau threefolds, such that any K\"ahler (thus, projective) Calabi--Yau threefold is the small resolution
of a degeneration of this family and that any two K\"ahler Calabi--Yau threefolds may be related by deformations, small resolutions and their inverses through non-K\"ahler Calabi--Yau threefolds, although they are non-homeomorphic (\cite{Re}).
--- figuratively speaking, projective Calabi--Yau threefolds may appear as scattered islands in the sea of (generically non-projective) Calabi--Yau threefolds just as projective $K3$ surfaces appear as scattered islands in the sea of (generically non-projective) $K3$ surfaces. This speculation demonstrates a role of non-K\"ahler Calabi--Yau manifolds in understanding K\"ahler ones.

In dimension four, the situation is even more obscure.
A huge number of non-homeomorphic K\"ahler Calabi--Yau fourfolds can be constructed as complete intersections in toric varieties. However,  to the best of  the author's knowledge, not a single example of non-K\"ahler Calabi--Yau fourfold has been found.
The  purpose of this note is to construct an example of non-K\"ahler Calabi--Yau fourfold. Hence, we establish the following:
\begin{theorem}
There exists a non-K\"ahler Calabi--Yau fourfold.
\end{theorem}
We note that some examples of simply-connected non-K\"ahler compact holomorphic symplectic fourfolds 
 with trivial canonical class and \mbox{$ H^2(M, \mco_M) \neq 0$} were constructed (\cite {Bo, Gu}). We also remark that interests in non-K\"ahler Calabi--Yau manifolds in other directions also have been growing rapidly (\cite{ Fe,  FuLiYa, QiWa, To}).

We shall construct our example by smoothing a normal crossing variety.
By smoothing, we mean the reverse process of the semistable degeneration of a manifold to a normal crossing variety.
If a normal crossing variety is the central fiber of a semistable degeneration
of Calabi--Yau manifolds, it can be regarded as a member in a deformation
family of those Calabi--Yau manifolds.
 So building a normal crossing variety smoothable to
a Calabi--Yau manifold can be regarded as building a deformation type of
Calabi--Yau manifolds. The construction by smoothing is intrinsically up to deformation.

The structure of this note is as follows.

We start  Section \ref{sec2} by introducing two background materials -- a smoothing theorem  and a Calabi--Yau threefold.
The smoothing theorem will be used to smooth a normal crossing variety to a non-K\"ahler Calabi--Yau fourfold and the Calabi--Yau threefold will be used as a building block in constructing the normal crossing variety in the next sections.

In Section \ref{sec3}, we build a smoothable normal crossing variety.
The construction, starting from the Calabi--Yau threefold introduced in Section \ref{sec2}, involves several steps of taking quotients and blow-ups of varieties.

In Section \ref{sec4}, we showed that the fourfold, which is a smoothing of the normal crossing variety constructed in Section \ref{sec3}, is a non-K\"ahler Calabi--Yau fourfold and calculate its topological Euler number.

After finishing the manuscript, the author received an e-mail from Taro Sano, informing that he constructed  non-K\"ahler Calabi--Yau manifolds of dimension higher than three by smoothing method (\cite{Sa}).

\section{A smoothing theorem and a Calabi--Yau threefold}
\label{sec2}

By a variety, we mean a reduced complex analytic space.
Let $\mathcal X =X_1 \cup X_2$ be a variety whose irreducible components are two smooth varieties $X_1$ and $X_2$.
$\mathcal X$ is called a \emph{normal crossing variety} if, near any point $p \in X_1 \cap X_2$, $\mathcal X$ is locally isomorphic to
$$\{(x_0,x_1,\cdots, x_n) \in \mbc^{n+1} | x_{n-1} x_n = 0 \}$$
 with $p$ corresponding to the origin and $X_1$, $X_2$ locally corresponding to the hypersurfaces $x_{n-1}=0,  x_n = 0$ respectively in $\mbc^{n+1}$.
 Note that the variety $D_\mcX:=X_1 \cap X_2$ is  smooth.
Suppose that there is  a proper map  $\varsigma : \mathfrak{X} \ra \Delta$ from a complex
manifold ${\mathfrak X}$ onto the unit disk $\Delta=\{t \in \mbc | \| t\| \leq 1 \}$ such that
the fiber ${\mathfrak X}_t = \varsigma^{-1}(t)$ is a smooth manifold for every $t \neq 0$ and ${\mathfrak X}_0 = \mathcal X$.
We denote a generic fiber ${\mathfrak X}_t$ ( $t \neq 0$) by $M_\mcX$ and we say that $\mathcal X$ is a semistable degeneration of a smooth manifold $M_\mcX$ and that $M_\mcX$ is a \emph{semistable smoothing} (simply smoothing) of $\mcX$.

We will use the following result from \cite{HaSa} and  \cite{FeFiRu} as a   generalization of  results in \cite{KaNa}.
\begin{theorem} \label{thm21} Let $\mathcal X =X_1 \cup X_2$ be a normal crossing variety of dimension $n \ge 3$ whose irreducible components are two smooth compact varieties $X_1$ and $X_2$ such that $D_{\mathcal X}$ is smooth. Assume that
\begin{enumerate}
\item $\omega_{\mathcal X} \simeq {\mathcal O_{\mathcal X}}$,
\item $H^{n-1}({\mathcal X}, {\mathcal O}_{\mathcal X}) =0$, $H^{n-2}(X_i, {\mathcal O}_{X_i}) =0$ for $i=1, 2$,
\item $N_{D_\mcX /X_1} \otimes N_{D_\mcX /X_2} $
 on $D_\mcX$ is trivial.
\end{enumerate}

Then $\mcX$ is smoothable to a  complex manifold $M_\mcX$.
\end{theorem}

This smoothing theorem was proved firstly in \cite{KaNa} with an assumption of  K\"ahlerness condition and the condition  was removed  in \cite {HaSa} and \cite{FeFiRu}.
We will also need the following lemma (Lemma 3.14 in \cite{HaSa}) in the proof of Theorem \ref{thm41}.

\begin{lemma}
\label{lem22}
With conditions in Theorem \ref{thm21}, assume further that $X_1, X_2$ are projective and $M_\mcX$, $D_\mcX$ are a projective Calabi--Yau $n$-fold and a  Calabi--Yau $(n-1)$-fold respectively.
Then there exists a big line bundle  on $\mcX$.
\end{lemma}

We will build a normal crossing variety and smooth it to a Calabi--Yau fourfold by applying Theorem \ref{thm21} and show, using Lemma \ref{lem22},  that  the Calabi--Yau fourfold is non-K\"ahler.
To make the normal crossing variety $\mcX = X_1 \cup X_2$, we need to construct the varieties $X_1, X_2$ first.
$X_1$ and $X_2$ will be built from Beauville's Calabi--Yau threefold and two involutions on it.

We briefly recall  Beauville's Calabi--Yau threefold.
Let $\zeta = e^{2\pi \frac{\sqrt{-1}}{3}}$. By $E_\zeta$, we denote the elliptic curve whose
period is $\zeta$ and by $E_\zeta^3/\langle \zeta \rangle $ the quotient of
the product manifold $E_\zeta^3$ by the scalar multiplication by
$\zeta$. Let
$$Q_0 =0,\,\, Q_1 = \frac{1-\zeta}{3},\,\, Q_2 =\frac{-1+\zeta}{3} $$
in $E_\zeta$. These are exactly the fixed points
of the scalar multiplication by $\zeta$ on $E_\zeta$. For $i, j, k=0, 1, 2$, let
$$Q_{i,j,k} = (Q_{i}, Q_{j}, Q_{k}) \in E_\zeta^3$$
and let $\overline Q_{i,j,k}$ be its image in
$E_\zeta^3/\langle \zeta \rangle $. Then $\overline Y =
E_\zeta^3/\langle \zeta \rangle $ has singularities of type
$\frac{1}{3}(1,1,1)$ at $\overline Q_{i,j,k}$'s and  the blow-up
$ Y \ra \overline Y$ at these 27
singular points gives a Calabi--Yau threefold $Y$. This is a K\"ahler
 Calabi--Yau threefold with Hodge numbers $h^{1,2}(Y)=0$ and $h^{1,1}(Y) = 36$, which  was originally found by Beauville (\cite{Be2}).

\section{Construction of a normal crossing variety $\mcX = X_1 \cup X_2$}
\label{sec3}

Consider  two $3 \times 3$ matrices in
${\rm GL}_3 \left (\mbz[\zeta] \right)$
$$A_1 =\left(
         \begin{array}{rrr}
           -1 & 0 & 0 \\
           0 & 1 & 0 \\
           0 & 0 & 1 \\
         \end{array}
       \right),
 \,\,\, A_2 =\left(
         \begin{array}{rrr}
           1 & 0 & 0 \\
           -1 & 0 & 1 \\
           1 & 1 & 0 \\
         \end{array}\right).$$
 Note that $A_i^2$ is the $3 \times 3$ identity matrix and $\det(A_i)=-1$ for $i=1, 2$.
The matrix $A_i$ induces an involution $\sigma_i$ on $E_\zeta^3$. Note that $\sigma_i^*$ acts as multiplication by $-1$ on $H^{3,0} \left(E_\zeta^3 \right)$.
Noting that the subgroup of ${\rm GL}_3 \left (\mbz[\zeta] \right)$ that is generated by  $A_1$, $A_2$ is {infinite}, we make the following remark which will play a key role in showing the non-K\"ahlerness in the proof of Theorem \ref{thm41}.
\begin{remark}
\label{rem31}
  The group of automorphisms of  $E_\zeta^3$ that is generated by $\sigma_1$, $\sigma_2$ is infinite.
\end{remark}

There is a unique involution $\rho_i$ on $Y$ such that the following diagram commutes:
$$
\begin{CD}
E_\zeta^3 	@>\sigma_i>> E_\zeta^3\\
@VVV 		@VVV\\
\overline Y @.	 \overline Y\\
@AAA 		@AAA\\
Y 	@>\rho_i>> Y
\end{CD}
$$

 We note that $\rho_i^*$ also acts as multiplication by $-1$ on $H^{3,0} (Y)$ and the fixed locus $S_i$ of $\rho_i$ is a smooth surface (a disjoint union of  irreducible surfaces).
Now we move on to the construction of smooth varieties $X_1, X_2$.
We borrow a method of  construction from \cite{KoLe} (\S 4 in \cite{KoLe}) that makes use of non-symplectic involutions on $K3$ surfaces.

Let $\psi:\mbp^1 \ra \mbp^1$ be any  involution fixing two distinct
points  and consider a quotient
$\overline X_i = (Y \times \mbp^1) / (\rho_i \times \psi$). Then the singular locus of $\overline X_i$ is a
product of  smooth surfaces and ordinary double points, resulting from the
fixed locus of $\rho_i$. Let $ X_i \ra \overline X_i$ be the blow-up along the
singular locus of $\overline X_i$. It is elementary to check that ${X_i}$ is
smooth.  Choose a point $p \in \mbp^1$ such that $p\neq \psi(p) $. Let $D'_i$
be the image of $ Y \times \{p\}$ in $\overline X_i$ and $D_i$ be the inverse image of $D'_i$
in ${X_i}$. Then $D_i$ is isomorphic to $Y$ and it is an anticanonical
divisor of ${X_i}$ whose normal bundle $N_{D_i/{X_i}}$ in $X_i$ is
trivial. Since $Y$ is projective, all the varieties $Y\times \mbp^1$, $\overline X_i$ and $X_i$ are projective.

We summarize our notations, including ones to be defined:
\begin{itemize}
\item $\zeta = e^{2\pi \frac{\sqrt{-1}}{3}}$.
\item $E_\zeta = \mbc/(\mbz \oplus \mbz \zeta)$ is the elliptic curve with period $\zeta$.
\item $\overline Y = E_\zeta^3 / \langle \zeta \rangle$.
\item $\phi: \widetilde Y \ra E_\zeta^3$ is the blow-up at the 27 points of  $Q_{i,j,k}$'s.
\item $\eta: \widetilde Y \ra Y$ is the map induced by $E_\zeta^3 \ra \overline Y$.
\item $Y \ra \overline Y$ is the blow-up of $\overline Y$ at its singular points. $Y$ is a K\"ahler Calabi--Yau threefold.
\item $\sigma_i:E_\zeta^3 \ra E_\zeta^3$ is the involution induced by the matrix $A_i$ for $i=1, 2$.
\item $\rho_i:Y \ra Y$ is the involution induced by $\sigma_i$  for $i=1, 2$.
\item $S_i = Y^{\rho_i}$ is the fixed locus of $\rho_i$  for $i=1, 2$.
\item $\check S_i = \phi(\eta^{-1}(S_i))$, the fixed locus of $\sigma_i$ for $i=1, 2$.
\item $\psi:\mbp^1 \ra \mbp^1$ is an involution fixing two distinct points $q_1, q_2 \in \mbp^1$.
\item $\overline X_i = (Y \times \mbp^1) / (\rho_i \times \psi)$  for $i=1, 2$.
\item $X_i \ra \overline X_i$ is the blow-up along the singular locus of $\overline X_i$  for $i=1, 2$.
\item    $\widetilde {X_i} \ra Y\times\mbp^1$ is  the blow-up
of $Y\times\mbp^1$ along the surface $S_i \times \{q_1, q_2\}$.
\item $D'_i$: the image of $Y \times \{ p\}$ in $\overline X_i$ for a point $p \in \mbp^1$ with $p \neq \psi(p)$  for $i=1, 2$.
\item $D_i$ : the inverse image of $D'_i$ in $X_i$  for $i=1, 2$. $D_i \simeq Y$ and $D_i \in \left| -K_{X_i} \right|$.
\item    $X_i^* = X_i -D_i$.
\item $\mcX = X_1 \cup X_2$ is the normal crossing variety of $X_1, X_2$, made by gluing along their isomorphic smooth anticanonical sections $D_1, D_2$.
\item $D_\mcX = X_1\cap X_2$.
\item $M_{\mathcal X}$ is a smoothing of a normal crossing variety $\mathcal X = X_1 \cup X_2$.
\end{itemize}

We make a  normal crossing variety $\mcX = X_1 \cup X_2$ by gluing transversally  along $D_1$ and $D_2$, (see \S2 (especially Corollary 2.4) of \cite{HaSa} for details of the gluing process). Then $D_\mcX := X_1\cap X_2$ is a copy of $Y$. Since $D_\mcX = X_1 \cap X_2$ is an anticanonical divisor  of both $X_1$ and $X_2$, $\omega_\mcX \simeq \mco_\mcX$.
 Note that   $N_{D_\mcX /X_1} \otimes N_{D_\mcX /X_2} $ is trivial. Let $X_i^* = X_i -D_i$.
For varieties $X_1$, $X_2$ and $\mcX = X_1 \cup X_2$ constructed in this section, we gather some of their properties:
\begin{proposition}
\label{prop32}
\begin{enumerate}
\item \label{c_proj} $X_1$ and $X_2$ are projective.
\item \label{c_omega}  $\omega_{\mathcal X} \simeq {\mathcal O_{\mathcal X}}$.
\item \label{c_dsemi}  $N_{D_\mcX /X_1} \otimes N_{D_\mcX /X_2} $  on $D_\mcX$ is trivial.
\item \label{c_sim}  Both $X_i$ and $X_i^*$ are simply-connected for $i=1, 2$.
\item \label{c_comx}  $H^{k}(X_i, {\mathcal O}_{X_i})  =0$ for $i=1, 2$, $k=1,2,3,4$.
\item \label{c_commcx}  $H^{k}({\mathcal X}, {\mathcal O}_{\mathcal X}) =0$ for $k=1, 2, 3$.

\end{enumerate}
\end{proposition}
\begin{proof}
The properties (\ref{c_proj}), (\ref{c_omega}), (\ref{c_dsemi}) are already shown.

\medskip

We show the property (\ref{c_sim}).
One can obtain ${X_i}$ differently. Let ${q_1, q_2}$ be the
fixed points of $\psi$. Let $\widetilde {X_i}$ be the blow-up
of $Y\times\mbp^1$ along the surface $S_i \times \{q_1, q_2\}$. Then the involution on
$Y\times\mbp^1$ induces an involution on $\widetilde {X_i}$, whose fixed locus
is the exceptional divisor over $S_i \times \{q_1, q_2\}$. The quotient of~
$\widetilde {X_i}$ by the involution is isomorphic to ${X_i}$. This may
be summarized by the diagram,
$$
\begin{CD}
\widetilde{{X_i}} 	@>>> {{X_i}}\\
@VVV 		@VVV\\
Y\times\mbp^1 	@>>> \overline X_i
\end{CD}
$$

Let $\widehat {X_i} \ra {X_i}$ be the universal covering. The fourfold
$\widetilde {X_i}$ is simply-connected, therefore the quotient map
$\widetilde {X_i}\to{X_i}$ lifts to a  map $\widetilde{X_i} \ra \widehat{X_i}$
so there is a commutative diagram:
\[
\xymatrix{
    &   \widehat {X_i} \ar[d]\\
\widetilde {X_i} \ar[r]\ar[ur] & {X_i} .}
\]
Since the map $\widetilde {X_i} \ra  {X_i}$ is a double covering, the degree of the map $\widetilde {X_i} \ra \widehat {X_i}$ is one or two. Suppose that the degree of this map is one. Then it is an isomorphism and so the universal covering $\widehat {X_i} \ra {X_i}$ is of degree two, having the same branch locus with that of the map
 $\widetilde {X_i} \ra {X_i}$.
Since the involution $\rho_i \times \psi$ has fixed points, the induced involution on $\widetilde {X_i}$ also has fixed points and hence the branch locus of the map
 $\widetilde {X_i} \ra {X_i}$ is not empty. This means that  the universal covering $\widehat {X_i} \ra {X_i}$  has  non-empty branch locus, which is impossible. Hence, we conclude that  the degree of the map $\widetilde {X_i} \ra \widehat {X_i}$ is two and, accordingly, the universal covering $\widehat {X_i} \ra {X_i}$
is of degree one, which implies that it is an isomorphism.
Therefore, $\widehat {X_i}$ is necessarily isomorphic to ${X_i}$ and so ${X_i}$ is simply-connected.

There is a natural projection: $\overline X_i = (Y\times\mbp^1)/ ({\rho_i}\times \psi) \ra Y /{\rho_i}$.
Let $\nu$ be the composition of ${X_i} \ra \overline X_i$ and $(Y\times\mbp^1)/ ({\rho_i} \times \psi ) \ra Y /{\rho_i}$. Let $x \in Y/{\rho_i}$ be a point in the branch locus of
the map $Y \ra Y /{\rho_i}$. Then $\nu^{-1}(x)$ is a union of three smooth
rational curves, one of which (denoted by $l$) crosses ${D_i}$ transversely at
a single point and the other two are disjoint from ${D_i}$, resulting from the
blow-up. Since ${X_i}$ and ${D_i}$ are simply connected, the fundamental group $\pi_1(X_i^*)$ of $X_i^*$ is
generated by a loop around ${D_i}$. We can assume that the loop is contained in
$l^* = l - {D_i}$. Since the loop can be contracted to a point in $l^*$,
$X_i^*$ is simply-connected.

\medskip

We move on to the property (\ref{c_comx}).
Since
$$\dim H^k(X_i, \mco_{X_i}) \le \dim H^k(\widetilde {X_i}, \mco_{\widetilde {X_i}}) = \dim H^k(Y \times \mbp^1, \mco_{Y \times \mbp^1}) =0$$
for $k=1, 2$, we have
$$H^k(X_i, \mco_{X_i}) =0$$
for $k=1, 2$.

Note that $X_i$ has an effective anticanonical divisor $D_i$, which is a Calabi--Yau threefold. Hence, we have
$$H^4(X_i, \mco_{X_i}) \simeq H^0(X_i, \Omega_{X_i}^4) =0.$$

Taking the cohomology of the structure sheaf
sequence,
\[
0 \ra \mco_{X_i}(K_{X_i}) \ra \mco_{X_i} \ra \mco_{D_i} \ra 0,
\]
we obtain an exact sequence
\begin{align*}
  H^3({X_i}, \mco_{X_i}(K_{{X_i}}))   \ra H^3({X_i}, \mco_{X_i}) \ra  H^3({D_i}, \mco_{D_i}) \ra \\
  H^4({X_i}, \mco_{X_i}(K_{{X_i}})) \ra H^4({X_i}, \mco_{X_i}) =0.
\end{align*}
Since, by Serre duality,
$$H^3({X_i}, \mco_{X_i}(K_{{X_i}}))\simeq H^{1}({X_i}, \mco_{X_i}) =0,$$
$$\dim H^4({X_i}, \mco_{X_i}(K_{{X_i}})) = \dim H^0({X_i}, \mco_{X_i})=1$$ and
$$\dim H^3({D_i}, \mco_{D_i})=1,$$
we have $\dim   H^3({X_i}, \mco_{X_i}) =0$.

\medskip

Finally, we show the property (\ref{c_commcx}).
From the exact sequence of sheaves
$$0 \ra \mco_\mcX \ra \mco_{X_1}\oplus\mco_{X_2} \ra \mco_{D_\mcX} \ra 0,$$
we obtain an exact sequence
$$H^{k-1}({D_\mcX}, \mco_{D_\mcX}) \ra H^k({\mcX}, \mco_{\mcX})  \ra H^k({X_1}, \mco_{X_1}) \oplus H^k({X_2}, \mco_{X_2}).$$
Since
$$H^{k-1}({D_\mcX}, \mco_{D_\mcX}) =H^k({X_1}, \mco_{X_1}) = H^k({X_2}, \mco_{X_2}) =0 $$
for $k=2,3$, we have
$$H^2({\mcX}, \mco_{\mcX}) =H^3({\mcX}, \mco_{\mcX})  =0.$$
Moreover, the exact sequence
\begin{align*}
0 \ra H^0({\mcX}, \mco_{\mcX})  \ra H^0({X_1}, \mco_{X_1}) \oplus H^0({X_2}, \mco_{X_2}) \ra H^{0}({D_\mcX}, \mco_{D_\mcX})\\
 \ra H^1({\mcX}, \mco_{\mcX})  \ra H^1({X_1}, \mco_{X_1}) \oplus H^1({X_2}, \mco_{X_2})=0
\end{align*}
gives $H^1({\mcX}, \mco_{\mcX})  =0$.

\end{proof}

\section{The example}
\label{sec4}

By Theorem \ref{thm21} with the properties in Proposition \ref{prop32}, one can show that the normal crossing variety $\mcX$, constructed in Section \ref{sec3}, is smoothable to a smooth fourfold $M_\mcX$ with trivial canonical class.
We can also check that $H^i(M_\mcX,\mco_\mcX) = 0$ for $i = 1, 2,3$  by the upper semicontinuity theorem with the property (\ref{c_commcx}) in Proposition \ref{prop32}.

\begin{theorem}\label{thm41}
$M_\mcX$ is a non-K\"ahler Calabi--Yau fourfold.
\end{theorem}
\begin{proof}
We only need to show that $M_\mcX$ is  simply-connected and non-K\"ahler.

First, we show the simply-connectedness.
We can obtain the topological type of $M_\mcX$ by pasting $X_1^*$ and $X_2^*$.
 One can
regard the normal bundle $N_{{D_\mcX}/X_i}$ as a complex manifold containing $D_\mcX$. Then
$$N_{{D_\mcX}/X_i}^* :=  N_{{D_\mcX}/X_i} - D_\mcX$$
 is a $\mbc^*$-bundle over ${D_\mcX}$, where $\mbc^* :=\mbc - \{0\}$. The triviality property on $N_{D_\mcX /X_1} \otimes N_{D_\mcX /X_2} $  implies the map
 $$\varphi: N_{{D_\mcX}/X_1}^* \ra N_{{D_\mcX}/X_2}^*,$$
locally defined by
 $$(x \in \mbc^* , y \in {D_\mcX}) \mapsto \left({1}/{x} , y \right),$$
 is globally well-defined and an isomorphism. Note that ${D_\mcX}$ in $X_i$ has a neighborhood $U_i$ that is homeomorphic to $N_{{D_\mcX}/X_i}$.
 Let $U_i^* = U_i - {D_\mcX}$. Then the map $\varphi$ induces a homeomorphism between $U_1^*$ and $U_2^*$.
 One can construct a manifold $M'$ by pasting together $X_1^*$ and $X_2^*$ along $U_1^*$ and $U_2^*$ with the homeomorphism.
The manifold $M'$ is homeomorphic to $M_\mcX$.
 Note that $X_1^*$, $X_2^*$ are  simply-connected (the property (\ref{c_sim}) in Proposition \ref{prop32}).
 Hence, by Seifert--van Kampen theorem, $M'$ is simply-connected.

\medskip

For the non-K\"ahlerness, suppose that $M_\mcX$ is K\"ahler, then it is necessarily projective. Note that $D_\mcX$, $X_1$ and $X_2$ are all projective (the property (\ref{c_proj}) in Proposition \ref{prop32}). Hence, by Lemma \ref{lem22}, there exists a big line bundle $\mathcal L$ on $\mcX$. Let $h_i$ be the  big divisor class in $\Pic(X_i)$ corresponding to $\mathcal L|_{X_i}$, then $h_1|_{D_\mcX}$ is linearly equivalent to $h_2|_{D_\mcX}$. Note that $D_\mcX$ is a copy of $Y$.
Let us denote the divisor class in $\Pic(Y)$ of $h_1|_{D_\mcX}$, $ h_2|_{D_\mcX}$ by $\hat h$.
Chasing the construction of $X_1$ and $X_2$, one can check that $\hat h$ belongs to $\Pic(Y)^{\rho_1^*} \cap \Pic(Y)^{\rho_2^*}$, where $\Pic(Y)^{\rho_i^*}$ is the subgroup of $\Pic(Y)$ that consists of the classes invariant under ${\rho_i^*}$.

The linear system $\left | D_\mcX \right |$ is base-point free and it gives a fibration
$X_i \ra \mbp^1$
and $D_\mcX$ is one of its generic fibers. Hence $\hat h$ is a big divisor of $Y$ (see, for example, Corollary 2.2.11 in \cite{La}).
Let $\phi: \widetilde Y \ra E_\zeta^3$ be the blow-up at the 27 points of  $Q_{i,j,k}$'s and $\eta: \widetilde Y \ra Y$ be the map induced by $E_\zeta^3 \ra \overline Y$ such that the diagram commutes:
$$
\begin{CD}
{\widetilde Y} 	@>\eta>> {Y}\\
@V\phi VV 		@VVV\\
E_\zeta^3  	@>>> \overline Y
\end{CD}
$$

It is not hard to check that $\check h =\phi_*(\eta^*(\hat h))$ is a big divisor of $E_\zeta^3$ and the class $\check h$ belongs to
${\rm NS}(E_\zeta^3 )^{\sigma_1^*} \cap {\rm NS}(E_\zeta^3 )^{\sigma_2^*}$.
Note that any big divisor of the abelian variety $E_\zeta^3$ is ample.
However, the group of automorphisms of  $E_\zeta^3$ that is generated by $\sigma_1, \sigma_2$ is infinite (Remark \ref{rem31}) and so ${\rm NS}(E_\zeta^3 )^{\sigma_1^*} \cap {\rm NS}(E_\zeta^3 )^{\sigma_2^*}$ does not contain an ample class. Therefore, we have a contradiction and $M_\mcX$ should be non-K\"ahler.

%

\end{proof}

Topological invariants of $M_\mcX$ can be calculated from  the topological manifold $M'$ in the proof of Theorem \ref{thm41}.
For example, the topological Euler number $\chi(M_\mcX)$ of $M_\mcX$ is
$$\chi(M_\mcX) = \chi(M') = \chi(X_1^*) + \chi(X_2^*)- \chi(U_1).$$
Note
$$\chi(X_i) = \chi(X_i^*) + \chi(D_i) = \chi(X_i^*) + \chi(Y)$$
and $\chi(U_1) =\chi(S^1) \chi(D_1) = 0$. Hence,
$$\chi(M_\mcX) = \chi(X_1) + \chi(X_2)-2 \chi(Y).$$
The topological Euler characteristic of ${\widetilde X_i}$ is
\[\chi({\widetilde X_i}) = 2 \chi(Y)  + 2  \chi(S_i)\]
On the other hand,
\[2 \chi({X_i}) -4\chi(S_i)= \chi({\widetilde X_i})\]
and so
\begin{align*}
\chi({X_i}) &=\frac{1}{2} \left (4 \chi(S_i) + \chi({\widetilde X_i}) \right )\\
       &=2  \chi(S_i)  + \chi(Y)  +   \chi(S_i)\\
       &=\chi(Y) + 3  \chi(S_i).
\end{align*}

Let $\check S_i = \phi(\eta^{-1}(S_i))$.
Note that $\check S_i$ is the fixed locus of $\sigma_i$.
Let $\Theta = \{Q_{i,j,k} \mid i,j,k=0,1,2 \}$. One can easily check
$$\chi(S_i) = 2 \left | \check S_{i} \cap \Theta \right | = 18,$$
where $\left | \check S_{i} \cap \Theta \right |$ is the number of points in $ \check S_{i} \cap \Theta$.

Therefore, the topological Euler number $\chi(M_\mcX)$ of $M_\mcX$ is
\begin{align*}
\chi(M_\mcX) &= \chi(X_1) + \chi(X_2)-2 \chi(Y)\\
&= \chi(Y) + 3  \chi(S_1) +\chi(Y) + 3  \chi(S_2) - 2 \chi(Y)\\
&=   3  \chi(S_1) +3  \chi(S_2)\\
&=108.
\end{align*}

The pair of matrices  $A_1$, $A_2$ in Section \ref{sec2}, which can be used in the construction, is  obviously not unique.
Any pair of $3\times3$ matrices $A_1$, $A_2$ that satisfies the following conditions,
\begin{enumerate}
\item $A_i \in {\rm GL}_3(\mbz[\zeta])$,
\item $A_i^2$ is the $3 \times 3$ identity matrix,
\item $\det(A_i)=-1$,
\item $\rho_i$ is not fixed-free and its fixed locus is a smooth surface and
\item The subgroup of ${\rm GL}_3(\mbz[\zeta])$ containing both $A_1, A_2$ is infinite,
\end{enumerate}
gives rise to a non-K\"ahler Calabi--Yau fourfolds through the construction of Sections \ref{sec3} and \ref{sec4}, where the conditions guarantee, respectively,
\begin{enumerate}
\item $A_i$ induces an automorphism of $E_\zeta^3$ which  induces an automorphism of $Y$.
\item $A_i$ induces an involution of $E_\zeta^3$ which  induces an involution of $Y$,
\item  $\rho_i^*$  acts as multiplication by $-1$ on $H^{3,0} (Y)$ so that $X_i$ has an anticanonical section isomorphic to $Y$,
\item both $X_i$ and $X_i^*$ are simply-connected, which leads to the simply-connectedness of $M_\mcX$
and
\item the group of automorphisms of  $E_\zeta^3$ that is generated by $\sigma_1$, $\sigma_2$ is infinite, which eventually leads to the non-K\"ahlerness of $M_\mcX$.
\end{enumerate}
The author  could not obtain other non-K\"ahler Calabi--Yau fourfolds of different topological Euler numbers although he tried many pairs of such  matrices. The author suspects that all the pairs of such  matrices may give rise to non-K\"ahler Calabi--Yau fourfolds of the same  topological Euler number ($=108$).


\end{document}